\documentclass[a4paper,8pt,oneside,bookmarks]{article}
\usepackage[utf8x]{inputenc}
\usepackage{graphicx}
\usepackage{wrapfig}
\usepackage[pdftex,bookmarks]{hyperref}
\hypersetup{
pdftitle={Randomized versions of Mazur lemma and  Krein–\v{S}mulian theorem with application to conditional convex risk measures for portfolio vectors},
pdfauthor={José M. Zapata}
}

\usepackage[english]{babel}
\usepackage{amsmath}
\usepackage{amsthm}
\usepackage{amsfonts}

\usepackage{amscd}

\usepackage{epigraph}

\usepackage{amssymb}

\usepackage{mathrsfs} 

\newcommand{\N}{\mathbb{N}}
\newcommand{\R}{\mathbb{R}}
\newcommand{\E}{\mathbb{E}}
\newcommand{\PP}{\mathbb{P}}

\newcommand{\F}{\mathcal{F}}
\newcommand{\A}{\mathcal{A}}
\newcommand{\EE}{\mathcal{E}}

\newcommand{\Vertiii}[1]{{\left\vert\kern-0.25ex\left\vert\kern-0.25ex\left\vert #1 
    \right\vert\kern-0.25ex\right\vert\kern-0.25ex\right\vert}}

\DeclareMathOperator*{\esssup}{ess.sup\,}
\DeclareMathOperator*{\essinf}{ess.inf\,}

\DeclareMathOperator*{\interior}{int\, }

\newtheorem{defn}{Definition}[section]
\newtheorem{rem}{Remark}[section]
\newtheorem{prop}{Proposition}[section]

\newtheorem{lem}{Lemma}[section]
\newtheorem{thm}{Theorem}[section]
\newtheorem{cor}{Corollary}[section]
\newtheorem{example}{Example}[section]
\providecommand{\keywords}[1]{\textbf{\textit{Keywords:}} #1}

\begin{document}

%%%%%%%%%%%%%%%%%%%%%%%%%%%%%%%%%%%%%%%%%%%%%%%%%%%%%%%%%%%%

\title{Randomized versions of Mazur lemma and  Krein–\v{S}mulian theorem}

\author{José M. Zapata \thanks{Universidad de Murcia, Dpto. Matemáticas, 30100 Espinardo, Murcia, Spain, e-mail: jmzg1@um.es}\thanks{The author is grateful to an anonymous referee for a careful review of the manuscript.}}

%\date{\today}
\maketitle

\begin{abstract}
We  extend to the framework of locally $L^0$-convex modules some results from the classical convex analysis. Namely, randomized versions of Mazur lemma and Krein-\v{S}mulian theorem under mild stability properties are provided.

\end{abstract}


%Mathematics Subject Classification (2010) 46H25 · 47H40 · 52A41 · 91B30 · 91B70

%JEL Classification G3 · D5

%\keywords{representation; $L^0$-modules; Fatou property; $L^\infty$-type module; Krein-Smulian theorem.}
\section*{Introduction}
\label{intro}

Over the past seventeen years, models of risk have gained increasing importance in the literature cf.\cite{key-2,key-7,key-5,key-3,key-4}. In particular, dynamic and conditional measurement of risk play an important role in recent works cf.\cite{key-7,key-6,key-9,key-17}. Mainly, this is because, in contrast to the static case, it allows a precise and consistent measurement of the risk of financial positions over time, taking into account the arrival of new information throughout the strategy of risk management.

An outstanding approach to this type of risk measures was provided by Filipovic et al.\cite{key-8} (see also \cite{key-9}). They considered a probability space $(\Omega,\F,\PP)$ which models the market information at some time $t$ where the risk management decisions  are taken, and a module over $L^0(\mathcal{F})$ (the space of classes of equivalence of $\F$-measurable functions) which models the value of the different financial position at some time horizon $T>t$.  This gave rise to the notion of locally $L^0$-convex module, which is a generalization of the classical notion of locally convex topology.

In the literature, we can find several works in which functional analysis results have been approached by considering modules over rings of random variables. For instance, randomly normed $L^0$-modules have been used as a tool for the study of ultrapowers of Lebesgue-Bochner function spaces by R. Haydon et al.\cite{key-31}. Further, it must also be highlighted the extensive research done by T. Guo, who has widely researched theorems from functional analysis under the structure of $L^0$-modules; firstly by considering the topology of stochastic convergence with respect to the $L^0$-seminorms cf.\cite{key-10,key-32,key-22,key-35}, and later the locally $L^0$-convex topology introduced by  Filipovic et al.\cite{key-8} and the connections between both,  cf.\cite{key-15,key-12,key-34,key-36}. It is also noteworthy that Eisele and Taieb \cite{key-18}  extended some theorems from functional analysis to modules over the ring $L^\infty$. 

Working with scalars into $L^0$ instead of $\R$ implies some difficulties. For example, $L^0$ neither is a field nor is endowed with a total order, further the locally $L^0$-convex topology lacks a countable neighborhood base of $0\in L^0$. Among other  difficulties, this is why arguments that are given to prove theorems from functional analysis often fail under the structure of $L^0$-module. For this reason, many works often consider additional 'stability conditions' or 'countable concatenation properties' on either the algebraic structure or the topological structure cf.\cite{key-8,key-15,key-11}. For instance, in \cite{key-11} it is introduced the notion of \textit{relatively stable subset} (in \cite{key-11} this kind of subset is called closed under countable concatenations) in order to provide a characterization of those topological $L^0$-modules whose topology is induced by a family of $L^0$-seminorms.

The purpose of this paper is to prove some relevant results from the functional analysis in the context of locally $L^0$-convex modules.
 First, we will prove a randomized version of Mazur's lemma for $L^0$-normed modules for which the operation sum preserves the relative stability of subsets. 

Second, we provide a randomized versions of Krein-\v{S}mulian theorem. It becomes clear from \cite{key-5} and \cite{key-1}, that some of the basic theorems of representation of risk measures or risk assessments rely on Krein-\v{S}mulian theorem.  Furthermore, this classical result plays a key role in the fundamental theorem of the asset pricing (see \cite{key-20}).  Hence we believe that this result may lead to financial applications in dynamic configurations of time. 

In the literature we have some related results: an earlier version of Mazur's  can be found in \cite[Corollary 3.4]{key-35} but in the topology of stochastic convergence. A Krein-\v{S}mulian theorem for $L^\infty$-modules can be found in \cite{key-18}. Also, a Krein-\v{S}mulian theorem for conditional locally convex spaces can be found in \cite{key-21} with a different strategy-proof and strong stability properties.

This manuscript is structured as follows: Section 1 is devoted to notation and preliminaries. In Section 2, we prove the randomized version of Mazur's lemma. Section 3 is devoted to prove the randomized version of Krein-\v{S}mulian.

\section{Preliminaries and notation}
\label{sec-1}

For the convenience of the reader, let us list some notation. Let be given a probability space $\left(\Omega,\mathcal{F},\PP\right)$, and let us consider $L^{0} \left(\Omega,\mathcal{F},\PP\right)$, or simply  $L^{0}$, the set of equivalence classes of  real valued $\mathcal{F}$-measurable random variables. It is known that the triple $\left(L^{0},+,\cdot\right)$ endowed with the partial order of the almost sure dominance is a lattice ordered ring. We will follow the common practice of identifying a random variable with its equivalence class. For given $\eta,\xi\in L^0$,  we will write ``$\eta\geq \xi$`` if $\PP\left( \eta\geq\xi\right)=1$, and likewise, we will write ``$\eta>\xi$'', if $\PP\left( \eta> \xi \right)=1$. %And, given $A\in \mathcal{F}$, we will write $x>y$ (respectively,  $x \geq y$) on $A$, if $\PP\left(x>y \mid A\right)=1$ (respectively , if $\PP\left(x \geq y \mid  A \right)=1$). 
We also define $L_{+}^{0}:=\left\{ \eta\in L^{0}\:;\: \eta\geq 0\right\}$ and $L_{++}^{0}:=\left\{ \eta\in L^{0}\:;\: \eta>0  \right\}$. %and $\mathcal{F}^+:=\{A\in \mathcal{F} ; \: \PP(A)>0 \}$. 
 We will denote by $\bar{L^{0}}$, the set of equivalence classes of  $\mathcal{F}$-measurable random variables taking values in $\overline{\R}=\R\cup\{\pm\infty\}$, and the partial order of the almost sure dominance is extended to $\bar{L^{0}}$ in a natural way. Furthermore, given a subset $H\subset L^{0}$, it is well-known that $H$ owns both an infimum and a supremum in $\bar{L^{0}}$ for the order of the almost sure dominance, which will be denoted by $\essinf H$ and $\esssup H$, respectively.  

This order also allows us to define a topology. We define $B_{\varepsilon}:=\left\{ \eta\in L^{0}\:;\:\left|\eta\right|\leq\varepsilon\right\}$ the ball of radius $\varepsilon\in L_{++}^{0}$ centered at $0\in L^{0}$. Then $\left\{ \eta+B_{\varepsilon}\:;\:\eta\in L^0,\:\varepsilon\in L_{++}^{0}\right\} $ is a base for a Hausdorff topology on $L^{0}$ (see Filipovic et al.\cite{key-8}).  $L^0[|\cdot|]$ stands for $L^0$ endowed with this topology.

We denote by $\A$ the measure algebra associated to $\F$, obtained by identifying two events of $\F$ if, and only if, the symmetric difference of which is $\PP$-negligible. Again, we identify an element of $\F$ with its equivalence class of $\A$. We also define $\A^+:=\left\{A\in\A\:;\: \PP(A)>0 \right\}$.

Given a family $\{A_i\}_{i\in I}$ in $\mathcal{A}$, its supremum is denoted by $\vee_{i\in I} A_i$ and its infimum by $\wedge_{i\in I} A_i$. For given $B\in\A$, we define the set of partitions of $B$, which is given by $\Pi(B):=\{ \{A_k\}_{k\in \N} \subset \A \:;\: \vee A_k=B\textnormal{, }A_i \wedge A_j = 0, \textnormal{ for all }i\neq j\textnormal{, }i,j\in \N  \}$.

%\section{locally $L^0$-convex modules and stability properties}

Let us recall some notions of the theory of locally $L^0$-convex modules:

\begin{defn}
\label{defn: L0module}\cite[Definition 2.1]{key-8}
A topological $L^{0}$-module $E\left[\mathscr{T}\right]$ is an $L^{0}$-module $E$ endowed with a topology $\mathscr{T}$ such that 
\begin{enumerate}
\item $E\left[\mathscr{T}\right]\times E\left[\mathscr{T}\right]\longrightarrow E\left[\mathscr{T}\right],\left(x,x'\right)\mapsto x+x'$
and
\item $L^{0}\left[\left|\cdot\right|\right]\times E\left[\mathscr{T}\right]\longrightarrow E\left[\mathscr{T}\right],\left(\eta,x\right)\mapsto {\eta}x$
\end{enumerate}
are continuous with the corresponding product topologies.
\end{defn}

\begin{defn}
\label{defn: conveL0mod}
\cite[Definition 2.2]{key-8}
A topology $\mathscr{T}$ on an $L^{0}$-module $E$ is said to be locally $L^{0}$-convex 
 if there is a neighborhood base $\mathscr{U}$ of $0\in{E}$  such that each $U\in\mathscr{U}$ is
\begin{enumerate}
\item $L^{0}$-convex, i.e. ${\eta}x+{(1-\eta)}y\in U$ for all $x,y\in U$
and $\eta\in L^{0}$ with $0\leq \eta\leq 1$;
\item $L^{0}$-absorbing, i.e. for all $x\in E$ there is a $\eta\in L_{++}^{0}$
such that $x\in {\eta}U$;
\item $L^{0}$-balanced, i.e. ${\eta}x\in U$ for all $x\in U$ and $\eta\in L^{0}$
with $\left|\eta\right|\leq 1$.
\end{enumerate}
In this case, $E\left[\mathscr{T}\right]$ is called a locally $L^0$-convex module. 
\end{defn}

\begin{defn}
\label{defn: seminorm}
\cite[definition 2.3]{key-8}
A function $\left\Vert \cdot\right\Vert :E\rightarrow L_{+}^{0}$
is an $L^{0}$-seminorm on $E$ if:
\begin{enumerate}
\item $\left\Vert {\eta}x\right\Vert =\left|\eta\right|\left\Vert x\right\Vert $
for all $\eta\in L^{0}$ and $x\in E$;
\item $\left\Vert x+y\right\Vert \leq\left\Vert x\right\Vert +\left\Vert y\right\Vert$, 
for all $x,y\in E.$
\end{enumerate}
If moreover, $\left\Vert x\right\Vert=0$ implies $x=0$, 
then $\left\Vert \cdot\right\Vert $ is an $L^{0}$-norm on $E$
\end{defn}

Let $\mathscr{P}$ be a family of $L^0$-seminorms on an $L^{0}$-module
$E$. Given $F$ a finite subset of $\mathscr{P}$ and $\varepsilon\in L_{++}^{0}$, 
we define 

\[
\begin{array}{cc}
U_{F,\varepsilon}:=\left\{ x\in E\:;\:\left\Vert x\right\Vert_F \leq\varepsilon\right\}, & \textnormal{ where }\left\Vert x\right\Vert_F:=\esssup\{\Vert x\Vert \:;\: \Vert\cdot\Vert\in F \}.
\end{array}
\]

Then $\mathscr{U}:=\left\{ U_{F,\varepsilon}\:;\:\varepsilon\in L_{++}^{0},\: F\subset \mathscr{P}\: {finite} \right\} $ is a neighborhood base of $0\in E$ for some locally $L^0$-convex topology $\mathcal{T}$, which is called the topology induced by $\mathscr{P}$ (see \cite{key-8}). $E$ endowed with this topology is denoted by $E\left[\mathscr{P}\right]$.

%~\\

\begin{defn}
Given a sequence $\{x_k\}$ in an $L^0$-module $E$ and a partition $\left\{ A_k\right\}\in \Pi(\Omega)$, an element $x\in E$ is said to be a \textit{concatenation} of $\left\{ x_{k}\right\}$ along $\left\{ A_k\right\}$ if $1_{A_k}x_k=1_{A_k}x$ for all $k\in\N$.
\end{defn}

The following example shows that, for given $\{A_k\}\in \Pi(\Omega)$ and $\{x_k\}\subset E$, there can be more than one concatenation of $\{x_k\}$ along $\{A_k\}$. %As far as the authors know there is no such an similar example in the literature. 

\begin{example}

Let us take $\Omega=(0,1)$, $\mathcal{F}=\mathcal{B}(\Omega)$ the $\sigma$-algebra of Borel, $A_{k}=[\frac{1}{2^k},\frac{1}{2^{k-1}})$ with $k\in\mathbb{N}$ and $\PP=\lambda$ the Lebesgue measure. We define in $L^0(\Omega,\mathcal{F},\PP)$ the following equivalence relation
\[
x\sim y \textnormal{ if } 1_{A_k} x= 1_{A_k}y \textnormal{ for all but finitely many }k\in\N.
\]
If we denote by $\overline{x}$ the equivalence class of $x$, we can define $\overline{x}+\overline{y}:=\overline{x+y}$ and $\eta \cdot\overline{x}:=\overline{\eta x}$, obtaining that $E:=L^0/\sim$ is an $L^0$-module.

Then, for $x\in L^0$, we have that $1_{A_k}\overline{x}=\overline{1_{A_k} x}=\overline{0}$. Hence, any element of $E$ is a  concatenation of $\{0\}_k$ along $\{A_k\}_k$.   
\end{example}   

For a given partition $\{A_k\}\in \Pi(\Omega)$ and a sequence $\{x_k\}\subset E$, if there exists an unique concatenation of $\{x_k\}$ along $\{A_k\}$ it will be denoted by the formal sum $\sum 1_{A_k} x_k$.

Now, we will collect some stability notions for $L^0$-modules, more of them contained  —under different names— in the existing literature, cf.\cite{key-8,key-15,key-11}.

\begin{defn}
Let $E$ be an $L^0$-module, then:
\begin{enumerate}

\item $K\subset E$ is said to be \textit{stable} (with uniqueness), if for each sequence $\left\{ x_{k}\right\}$ in $K$ and each partition $\left\{ A_{k}\right\}\in \Pi(\Omega)$, it holds that there exists (an unique) concatenation $x\in E$ of $\{x_k\}$ along $\{A_k\}$.

\item  $K$ is said to be \textit{relatively stable} (with uniqueness) if every concatenation $x$ of a sequence $\{x_k\}\subset K$ along a partition $\{A_k\}\in\Pi(\Omega)$ is again in $K$ (and any other concatenation of $\{x_k\}$ along $\{A_k\}$ equals $x$).
\end{enumerate}
\end{defn}

\begin{rem}
Note that the property of being relatively stable is exactly the same than the property of being \textit{closed under countable concatenations} introduced in \cite{key-11}; we only change the name of the property. In the property of being relatively stable we do not ask the existence of the concatenation, we only ask that if it exists then, it belongs to $K$. Clearly, if $K\subset E$ is stable, it is also relatively stable.   

Notice that every $L^0$-module $E$ is always relatively stable; however, we can find examples of $L^0$-modules which are not stable (see Example down below). In particular, we find that being relatively stable is a strictly weaker property than being stable. 

Inspection shows that every relatively stable subset $K$ of a stable $L^0$-module $E$ is also stable.
\end{rem}

\begin{example}
\cite{key-8}
\label{ex: ccp}
Consider the probability space $\Omega = [0, 1]$, $\mathcal{F} = \mathcal{B}[0, 1]$ the Borel $\sigma$-algebra and
$\PP:=\lambda$ the Lebesgue measure on $[0, 1]$ and define

\[
E := span_{L^0} \left\{1_{[1−\frac{1}{2^{n−1}},1−\frac{1}{2^n}]}\: ; \: n \in \N\right\}.
\]
$E$ is an $L^0$-module that is not stable, but it is relatively stable.
\end{example}

We have the following theorem, which characterizes the locally $L^0$-convex topologies that are induced by a family of $L^0$-seminorms.  

\begin{thm}
\label{thm: caracterizacion}
\cite[Theorem 2.6]{key-11}
Let $E\left[\mathscr{T}\right]$ be a topological $L^{0}$-module. Then $\mathscr{T}$ is induced by a family of $L^0$-seminorms if, and only if, there is a neighborhood base of $0\in{E}$ for which each $U\in\mathscr{U}$ is $L^{0}$-convex, $L^{0}$-absorbing, $L^{0}$-balanced and relatively stable.
\end{thm}

Further, Zapata \cite{key-11}, and independently Wu and Guo \cite{key-19}, provided an example showing that in the above theorem, the property of being relatively stable cannot be dropped (see \cite[Example 2.1]{key-11}).

Regarding again the uniqueness for concatenations of a sequence along a partition, we have the following result:

\begin{prop}
Let $E[\mathscr{T}]$ be a locally $L^0$-convex module, and suppose that $\mathscr{T}$ is induced by a family of $L^0$-seminorms and is Hausdorff. 
Then for given $\{A_k\}\in \Pi(\Omega)$ and $\{x_k\}\subset E$, if there exists a concatenation of $\{x_k\}$ along $\{A_k\}$, then there is exactly one of them.    
\end{prop}

\begin{proof}

Suppose that $\mathscr{U}$ is a neighborhood base of $0\in E$ as in Theorem \ref{thm: caracterizacion}. Since $\mathscr{T}$ is Hausdorff, by just following the argument of the classical case, we have that $\bigcap\mathscr{U}=\{0\}$. 

So, for given $x,y\in E$ concatenations of $\{x_k\}$ along $\{A_k\}$, it is clear that $x-y$ is a concatenation of the sequence $0,0,...$ along $\{A_k\}$. Thus, since each $U\in\mathscr{U}$ is relatively stable, it holds that $x-y\in U$ for all $U\in\mathscr{U}$. So, we conclude that $x=y$.  
\end{proof}

The \textit{stable hull} of a family of $L^0$-seminorms was introduced in \cite{key-8}. Since in \cite{key-8} families of $L^0$-seminorms are supposed to be invariant under finite suprema, we adopt the definition suggested by Guo et al.\cite{key-34}: 

\begin{defn}
Let $\mathscr{P}$ be a family of $L^0$-seminorms on an $L^{0}$-module $E$. We define the \textsl{stable} hull of $\mathcal{P}$ as follows:
\[
\begin{array}{cc}
\mathscr{P}_{cc}:=\left\{ \underset{k\in\N}\sum 1_{A_k}\Vert \cdot\Vert_{F_k} \:;\:{A_k}\in\Pi(\Omega),\: F_k\subset \mathcal{P} \textnormal{ finite} \right\}, & \textnormal{ with }\Vert x\Vert_{F_k}:=\esssup \left\{\Vert x\Vert  \:;\: \Vert\cdot\Vert\in F_k \right\}.
\end{array}
\]
\end{defn}

Filipovic et al. \cite{key-8} introduced the topological dual of a topological  $L^0$-module $E[\mathscr{T}]$, which is denoted by 
\[
E[\mathscr{T}]^*=E^*=\left\{ \mu :E\rightarrow L^0\:;\:\mu\textnormal{ is }L^0\textnormal{-linear and continuous}\right\}. 
\]

We can also consider the weak and weak-$*$ topologies. Indeed, let us consider the family of $L^0$-seminorms $\sigma(E,E^*)=\{ q_{x^*}\}_{x^*\in E^*}$ defined by $q_{x^*}(x):=|x^*(x)|$ for $x\in E$. Analogously, we have the weak-$*$ topology. 

Then, we can consider the stable hull $\sigma(E,E^*)_{cc}$, which induces a topology finer than the weak topology. Likewise, $\sigma(E^*,E)_{cc}$ induces a topology finer than the weak-$*$ topology.   

The topologies induced by  $\sigma(E,E^*)_{cc}$ and $\sigma(E^*,E)_{cc}$, will be referred to as the stable weak topology and the stable weak-$*$ topology, respectively.

The following example shows that both topologies are not necessarily the same.

\begin{example}
\cite[Example 1.2]{key-50}
\label{ex: weakTop}

Filipovic et al. \cite{key-8} introduced the following locally $L^0$-convex modules, which are called $L^p$-type modules. Namely, let $(\Omega,\EE,\PP)$ a probability space such that $\F$ is a sub-$\sigma$-algebra of $\EE$ and $p\in [1,+\infty]$. Then we can define the $L^0$-module  
$L^p_\F(\EE):=L^0(\F) L^p(\EE)$, for which
\[
\left\Vert x | \mathcal{F} \right\Vert _{p}:=\begin{cases}
\E_\PP\left[\left|x\right|^{p}|\mathcal{F}\right]^{1/p} & \textnormal{ if } p<\infty\\
\essinf\left\{ y\in\bar{L}^{0}\left(\mathcal{F}\right)\:;\: y\geq\left|x\right|\right\}  & \textnormal{ if } p=\infty
\end{cases}
\]
defines an $L^0$-norm.

Besides, it is known that for $1\leq p<+\infty$, if $1<q\leq+\infty$ with $1/p+1/q=1$, the map $T:L^p_\F(\EE)\rightarrow L^q_\F(\EE)$, $y\mapsto T_y$ defined by $T_y(x):=\E_\PP[x y|\F]$ is an $L^0$-isometric isomorphism (see \cite[Theorem 4.5]{key-15}).

The weak topologies are defined, and the family of sets 
\[
U_{F,\varepsilon}:=\left\{x\in L^p_\F(\EE)\:;\: |\E_\PP[x y|\F]|\leq\varepsilon,\:\forall y\in F\right\},
\]
where $F$ is a finite subset of $L^q_\F(\EE)$ and $\varepsilon\in L^0_{++}(\F)$, constitutes a neighborhood base of $0\in L^p_\F(\EE)$ for the weak topology.

Let us consider the particular case:  $\Omega=(0,1)$, $\mathcal{E}=\mathcal{B}(\Omega)$ the $\sigma$-algebra of Borel, $A_k=[\frac{1}{2^k},\frac{1}{2^{k-1}})$ with $k\in\mathbb{N}$, $\F:=\sigma(\{A_k\:;\:k\in\N\})$ and $\PP$ the Lebesgue measure. 

Also, for each $k\in\N$, let us define $L^2_k:=L^2(A_k)$ considering the trace of $\F$ in $A_k$ and the conditional probability $\PP(\cdot|A_k)$. For each $x\in L^2_k$ we denote $\Vert x|A_k\Vert_2:=\E_\PP[x^2|A_k]^{1/2}$.

In this case, $L^0(\F)=\{ \sum_{k\in\N}1_{A_k}r_k\:;\: r_k\in\R\}$, and inspection shows $L^2_\F(\EE)=\{\sum_{k\in\N}1_{A_k}x_k\:;\: x_k|_{A_k}\in L^2_k\}$ and, for $x\in L^2_\F(\EE)$, we have $\Vert x|\F\Vert_2=\sum_{k\in\N}1_{A_k}\Vert x|A_k\Vert_2$.

For each $k$, let us choose a countable set $\{y^k_n\:;\:n\in\N\}$ with $y^k_n|_{A_k}=z^k_n$ where $\{z^k_n\}$ is a linearly  independent subset of $L^2_k$.

Let
\[
U:=\sum_{k\in\N} 1_{A_k} U_{\{y^k_1,...,y^k_k\},1}.
\]
We have that $U$ is a neighborhood of $0\in L^2_\F(\EE)$ for the topology $\sigma(L^2_\F(\EE),L^2_\F(\EE))_{cc}$, but not for the topology $\sigma(L^2_\F(\EE),L^2_\F(\EE))$. Indeed, by way of contradiction, let us assume the there is a finite subset $F$ of $L^2_\mathcal{F}(\mathcal{E})$ and $\varepsilon\in L^0_{++}(\mathcal{F})$, where $\varepsilon=\sum_{k\in\N} 1_{A_k}r_k$ with $r_k\in\R^+$, so that $U_{F,\varepsilon}\subset U$. Now, let us take $k:=\#F + 1$. 

%Considering the random variables restricted to $A_k$ and considering the trace of $\F$ in $A_k$, we can suppose $\Omega=A_k$. Then

Let us define $F_k:=\{y|_{A_k}\:;\: y\in F\}$, then
\[
\begin{array}{cc}
U_{F_k,r_k}\subset U_{\{z^k_1,...,z^k_k\},1} & \textnormal{ in }L^2_k.
\end{array}
\]
For each $y\in L^2_k$, let us denote $\mu_y(x):=\E_{\PP}[x y|A_k]$. Then, it follows that
\[
\bigcap_{y\in F} \ker (\mu_y) \subset \bigcap_{i=1}^k \ker (\mu_{y^k_i}).
\]
But this is impossible, because $\bigcap_{y\in F} \ker (\mu_y)$ is a vectorial subspace of $L^2_k$ with codimension less than $k$, included into a vectorial subspace with codimension $k$. 
\end{example}

\section{A randomized version of Mazur lemma}
\label{sec-2}

A well-known result of classical convex analysis is the Mazur lemma, which shows that for any weakly convergent sequence in a Banach space there is a sequence of convex combinations of its members that converges strongly to the same limit. We  will generalize this tool to nets in an $L^0$-normed module. %For this purpose, we will consider not only $L^0$-convex combinations, but also countable concatenations of members of a given net.

%finding that it is not sufficient to take the convex hull, but it is necessary to take the countable concatenation closure.

Before proving this result, we need to recall the gauge function for $L^0$-modules, which was introduced in \cite[Definition 2.21]{key-8}:

Let $E$ be an $L^0$-module. The gauge function $p_{K}:E\rightarrow\bar{L}_{+}^{0}$ of
a set $K\subset E$ is defined by 

\[
p_{K}\left(x\right):=\essinf\left\{ \eta\in L_{+}^{0}\:;\: x\in {\eta}K\right\} .
\]

In addition,  among other properties, if $K$ is $L^{0}$-convex and $L^{0}$-absorbing then, the essential infimum above can be defined by taking $\eta\in L^0_{++}$, also $p_K$ is subadditive and $\eta p_K(x)=p_K(\eta x)$ for all $x\in E$ and $\eta\in L^0_+$. If, moreover, $K$ is $L^0$-balanced, then $p_{K}$ is an $L^0$-seminorm (see \cite[Proposition 2.23]{key-8}).

Given that $L^0$ is not a totally ordered set, the following result will be useful in some situations:  

\begin{lem}
\label{prop: cccdownwardsDirected}
Let $C\subset L^0$ be bounded below (resp. above) and relatively stable, then for each $\varepsilon\in L^0_{++}$ there exists $\eta_\varepsilon \in C$ such that 
\[
\essinf C \leq \eta_\varepsilon < \essinf C + \varepsilon
\]
 (resp., $\esssup C \geq \eta_\varepsilon > \esssup C - \varepsilon$) 

In particular, given an $L^0$-module $E$ and $K\subset E$, which is $L^0$-convex, $L^0$-absorbing and relatively stable, we have that, for $\varepsilon\in L^0_{++}$, there exists $\eta_\varepsilon \in L^0_{++}$ with $x \in {\eta_\varepsilon}K$ such that $p_K(x)\leq \eta_\varepsilon < p_K(x)+\varepsilon$. 
%\[
%p_K(X)\leq Y_\varepsilon < p_K(X)+\varepsilon
%\] 
\end{lem}  
\begin{proof}

Firstly, let us see that $C$ is downwards directed. Indeed, for given $\eta,\xi\in C$, define $A:=(\eta<\xi)$. Since $C$ is relatively stable, we have that $1_A \eta + 1_{A^c} \xi=\eta\wedge\xi \in C$. 

Therefore,  for $\varepsilon \in L^0_{++}$, there exists a decreasing sequence $\left\{\eta_k\right\}$ in $C$ converging  to $\essinf C$ almost surely. 

Let us consider the sequence in $\A$
\[
A_0:=\emptyset \textnormal{ and } A_k:=(\eta_k<\essinf C +\varepsilon)-A_{k-1} \textnormal{ for }k>0. 
\]
Then $\left\{A_k\right\}_{k\geq 0}\in \Pi(\Omega)$, and we can define $\eta_\varepsilon:=\sum_{k\geq 0}{1_{A_k}}\eta_k$. Given that $C$  is relatively stable, it follows that $\eta_\varepsilon \in C$.
%\[
%Y_\varepsilon:=\sum_{k\geq 0}{1_{A_k}}Y_k.
%\]
%Thus $Y_\varepsilon \in C$, since $C$ is stable under countable concatenations.

For the second part, it suffices to see that if $K$ is relatively stable, then $\left\{\eta\in L^0_{++}\:; \: x\in {\eta}K  \right\}$ is relatively stable as well. Indeed, given $\left\{A_k\right\}\in \Pi(\Omega)$ and $\left\{\eta_k\right\} \subset L^0_{++}$ such that $x\in {\eta_k}K$ for each $k\in\N$, let us take $\eta:={\sum}_{n\in\mathbb{N}}\eta_{n}1_{A_{n}}\in L_{++}^{0}$. Then we have that $x/\eta$ is a concatenation of $\left\{x/\eta_{k}\right\}$ along $\{A_{k}\}$. Since $x/\eta_{k}\in K$ for all $k\in\N$ and $K$ is relatively stable, we conclude that  $x/\eta\in K$, and the proof is complete.
%Therefore $X/Y\in K$ as $K$ is stable under countable concatenations, and the proof is complete.    
\end{proof}

\begin{rem}
Note that, in the latter result, we have considered $K$ to be relatively stable, rather than the stronger condition of being stable. This is because $\eta$, and therefore $x/\eta$, is constructed by using the stability of $L^0$ instead of some stability in $E$. Once $x/\eta$ exists, we just need to ensure that it belongs to $K$. But that is the case due to the relative stability of $K$. The same strategy was used to prove Proposition 2.3 of \cite{key-11}.
\end{rem}

Before stating the main result, we need to introduce some notions and prove some preliminary results:

\begin{defn}
Given $A\subset E$, we define: 
\begin{itemize}
	\item 

the $L^0$-\textit{convex hull} of $A$: 
\[
co_{L^0}(A):=\left\{\underset{i\in I}{\sum}{\eta_i x_i} \:; \: I \textnormal{ finite, }x_i \in A \textnormal{, } \eta_i \in L^0_+\textnormal{, }\underset{i\in I}{\sum}{\eta_i}=1 \right\}
\]
\item the \textit{stable hull} of $A$:
\[
\overline{A}^{cc}:=\left\{x\in E\:;\: \exists\{x_k\}\subset A,\:\{A_k\}\in \Pi(\Omega)\textnormal{ such that }\:x\textnormal{ is a concatenation of }\{x_k\}\textnormal{ along }\{A_k\}\right\}.
\]
\end{itemize}
\end{defn}

\begin{rem}
In many works the stable hull of a subset $K$ is defined assuming that $E$ is stable cf.\cite{key-15,key-12}. Notice that the definition provided here does not require such an assumption. 

Also note that $\overline{K}^{cc}$ is the least relatively stable set containing $K$. In particular, $K$ is relatively stable if, and only if, $\overline{K}^{cc}=K$. 

\end{rem}

%Now let us see the following proposition

\begin{prop}
\label{closureGauge}
Let $E\left[\mathscr{T}\right]$ be a topological $L^0$-module and 
let $C\subset E$ be $L^{0}$-convex and $L^{0}$-absorbing.  Then the following are equivalent: 
\begin{enumerate}
	\item $p_C:E\rightarrow L^0$ is continuous. 
	\item $0\in \interior{C}$.
\end{enumerate}

In this case, if in addition $C$ is relatively stable
\[
\overline{C}=\left\{x \in E \:; \: p_C(x)\leq 1  \right\},
\]
where $\overline{C}$ denotes the closure of $C$ with respect to $\mathscr{T}$.
\end{prop}

\begin{proof}

The proof is exactly the same as the real case.

For the equality.

``$\subseteq$'': It is obtained from continuity of $p_C$.

``$\supseteq$'': Let $x\in E$ be satisfying $p_C(x)\leq 1$. By Lemma \ref{prop: cccdownwardsDirected}, we have that for every $\varepsilon \in L^0_{++}$ there exists $\eta_\varepsilon \in L^0_{++}$ such that $1 \leq \eta_\varepsilon < 1+\varepsilon,$ with $x \in \eta_\varepsilon C$.

Then, $\left\{x/\eta_\varepsilon \right\}_{\varepsilon \in L^0_{++}}$ is a net (viewing $L^0_{++}$ downward directed) in $C$ converging to $x$ and therefore $x\in \overline{C}$. So, the proof is complete.
\end{proof}

The example below (drawn from \cite{key-11}) shows that, for the equality proved in the latter proposition, $C$ is required to be relatively stable.
\begin{example}
Consider the probability space  $\Omega=(0,1)$, $\mathcal{E}=\mathcal{B}(\Omega)$ the $\sigma$-algebra of Borel, $A_{n}=[\frac{1}{2^n},\frac{1}{2^{n-1}})$ with $n\in\mathbb{N}$, $\PP:=\lambda$ the Lebesgue measure and $E:=L^0(\mathcal{E})$ endowed with $|\cdot|$.

We define the set 

\[
U:=\left\{ \eta\in L^{0}\:;\:\exists\, I\subset\mathbb{N}\textnormal{ finite, }\left|\eta 1_{A_{i}}\right|\leq 1\:\forall\, i\in\mathbb{N}-I\right\}.
\]

Inspection shows that $U$ is
$L^{0}$-convex, $L^{0}$-absorbing and $\overline{U}=U$.

But, $p_{U}\left(x\right)=0$ for all $x\in{L^0}$, and so $\left\{ x\:;\: p_{U}(x)\leq 1\right\} = E$. Indeed, it suffices to show that $p_{U_1}\left(1\right)=0$, since $p_{U_1}$ is an $L^0$-seminorm. By way of contradiction, assume that $p_{U_1}\left(1\right)>0$. Then, there exists $m\in\N$ 
such that $\PP\left[\left(p_{U_1}\left(1\right)>0\right)\cap{A_m}\right]>0$.
Define $A:=\left(p_{U_1}\left(1\right)>0\right)\cap{A_m}$, $\nu:=\frac{1}{2 }(p_{U_1\left(1\right)}+1_{A^c})$, $\eta:=1_{A^c}+\nu{1_A}$ and $x:=1_{A^c}+\frac{1}{\nu}{1_A}$. 
Thus, we have $1=\eta{x}\in{\eta U_1}$ and $\PP\left(p_{U_1}\left(1\right)>\eta\right)>0$, a contradiction. 
%
%Therefore
%\[
%\left\{ X\in E;\: p_{U}(X)\leq 1\right\} = E.
%\]

\end{example}

We need a new notion:

\begin{defn}
Let $E$ be an $L^0$-module, we say that the sum of $E$ preserves the relative stability, if 
\[
L,M\subset E \textnormal{ are relatively stable implies that }L+M\textnormal{ is relatively stable.}
\]
\end{defn}

The following example shows that, in general, the sum of two relatively stable subsets is not necessarily relatively stable.

\begin{example}
\label{exam: counterExSPrccp}
Let us take $\Omega=(0,1)$, $A_{k}=[\frac{1}{2^k},\frac{1}{2^{k-1}})$ for each $k\in\mathbb{N}$, $\mathcal{F}=\sigma(\{A_k\:;\: k\in\N\})$ the sigma-algebra generated by $\{A_k\}_k$  and $\PP$ the Lebesgue measure restricted to $\mathcal{F}$. Then $(\Omega,\F,\PP)$ is a probability space. Let us put $L^0:=L^0(\Omega,\F,\PP)$.

Let us define the following $L^0$-module, which is an $L^0$-submodule of $L^0(\R^2,\F)$:
\[
E:=span_{L^0}\{(1,0),(0,1_{A_k}) \:;\: k\in\N\}. 
\]

Let us also consider the following subsets:
\[
L:=\overline{\{ (0,-1_{A_k}) \:;\: k\in\N\}}^{cc}
\]
\[
M:=\overline{\{ (1_{A_k}, 1_{A_k})  \:;\: k\in\N\}}^{cc}, 
\]
which are obviously relatively stable.

We also have $(1,0)\in E$ and 
\[
1_{A_k} (1,0)=1_{A_k}((0,-1_{A_k})+(1_{A_k},1_{A_k}))\in 1_{A_k}(L+M), 
\]
and therefore $(1,0)\in\overline{L+M}^{cc}.$

However, inspection shows that $(1,0)\notin L+M$. We conclude that $L+M$ is not relatively stable. 

\end{example}

The following result is easy to prove by inspection, we omit the proof:

\begin{prop}
\label{prop: sickModules}
Let $E$ be an $L^0$-module. If $E$ is stable, then its sum preserves the relative stability.
\end{prop}

In virtue of the above result we have that, for an $L^0$-module $E$, the property of being stable is stronger than the property of $E$ having sum which preserves the relative stability. The following result shows that, in fact, the former is strictly stronger than the latter.    

\begin{prop}
\label{prop: counterExSPrccp}
There exists an $L^0$-module $E$ which is not stable, but whose sum preserves the relative stability. 
\end{prop}
\begin{proof}
Let us revisit the $L^0$-module $E$ defined in Example \ref{ex: ccp}:
\[
\begin{array}{cc}
E := span_{L^0} \left\{1_{A_n}\: ; \: n \in \N\right\} &\textnormal{ with }A_n:=[1−\frac{1}{2^{n−1}},1−\frac{1}{2^n}]\textnormal{ for each }n\in\N.
\end{array}
\]
$E$ is an $L^0$-module that is not stable. Let us show that the sum of $E$ preserves the relative stability. Indeed, let us suppose that $L,M$ are relatively stable subsets of $E$.

Let $z\in E$ and $\{B_k\}_{k}\in\Pi(\Omega)$ be such that $1_{B_k}z=1_{B_k}(l_k+m_k)$ with $l_k\in L$ and $m_k\in M$ for each $k\in\N$.

First note that the following set
\[
F_L:=\left\{k\in\N \:;\: 1_{A_k} l\neq 0 \textnormal{ for all }l\in L\right\}
\]
is necessarily finite. Let us define $F_M$ in an analogous way, which is finite either.

Let us put 
\[
z=\underset{
i\in F,k\in\N
}{\sum} 1_{A_i \cap B_k} (l_k + m_k)=
\]
\[
\underset{
i\in F,k\in\N
}{\sum}1_{A_i \cap B_k} (l_k + m_k) +
\underset{
i\in F_L\cup F_M - F,\:
k\in\N
}{\sum}1_{A_i \cap B_k} (l_k + m_k)
=l + m,
\]
with 
\[
\begin{array}{cc}
l=\sum_{i\in F\cup F_L\cup F_M} \sum_{k\in\N} 1_{A_i \cap B_k} l_k, & m=\sum_{i\in F\cup F_L\cup F_M} \sum_{k\in\N} 1_{A_i \cap B_k} m_k.
\end{array}
\]
Since $L$ and $M$ are relatively stable, we conclude that $l\in L$ and $m\in M$. This shows that the sum preserves the relative stability
\end{proof}

Finally, we provide a version for $L^0$-modules of the classical Mazur lemma:

%\begin{lem}
%
%\end{lem}

\begin{thm}
\label{thm: Mazur}
[Randomized version of the Mazur lemma]
Let $(E,\left\Vert \cdot\right\Vert)$ be an $L^0$-normed module whose sum preserves the relative stability,\footnote[1]{In the earliest version of this paper, it was accidentally omitted the hypothesis on $E$ of having sum which preserves the relative stability. The author is grateful to Tiexin Guo for pointing out to him the possible lack of this condition, and for propitiating some interesting discussion which motivated Example \ref{exam: counterExSPrccp} and Proposition \ref{prop: counterExSPrccp}.} and let $\{x_\gamma\}_{\gamma\in \Gamma}$ be a net in $E$, which converges weakly to $x\in E$. Then, for any $\varepsilon\in L^0_{++}$, there exists
\[
z_{\varepsilon} \in \overline{co}^{cc}_{L^0}\{x_\gamma \:; \:\gamma\in \Gamma \} \textnormal{ such that }\left\Vert x - z_\varepsilon \right\Vert \leq \varepsilon.
\] 

%\[
%\left\Vert X - Z_\varepsilon \right\Vert \leq \varepsilon
%\]
\end{thm}

%\begin{rem} 
%In the earliest version of this paper, was accidentally omitted the hypothesis on $E$ of having sum which preserves the relative concatenation property. The author is grateful to Tiexin Guo for pointing out to him the possible lack of this condition. Also, some discussion with Tiexin Guo motivated Example \ref{exam: counterExSPrccp} and Proposition \ref{prop: counterExSPrccp}.

%Before proving this statement, we would like to highlight that we do not need to include the countable concatenation on $E$. We argue by way of contradiction, and the argument rely on the fact that the countable concatenation hull —which was defined without assuming the c. c. property on $E$— has the relative countable concatenation property. That is sufficient to reach a contradiction, by using the gauge function and the extension Hahn-Banach Theorem for $L^0$-modules, result that does not require the c. c. property on $E$ either. 
%\end{rem}

\begin{proof}
Define $M_1 := \overline{co}^{cc}_{L^0} \{x_\gamma \:; \: \gamma\in \Gamma \}.$ We may assume that $0\in M_1$, by replacing $x$ by $x-x_{\gamma_0}$ and $x_\gamma$ by $x_\gamma-x_{\gamma_0}$ for some $\gamma_0\in \Gamma$ fixed and all $\gamma\in\Gamma$.

By way of contradiction, suppose that for every $z \in M_1$ there exists $A_z \in \mathcal{A}^+$  such that $\left\Vert x - z \right\Vert > \varepsilon$ on $A_z$.
%\begin{equation}
%\label{ineqI}
%\begin{array}{cc}
%\left\Vert X - Z \right\Vert > \varepsilon & \textnormal{ on } A_Z 
%\end{array}
%\end{equation}

Denote $B_\frac{\varepsilon}{2}:=\{x\in E\:;\: \Vert x \Vert\leq \frac{\varepsilon}{2}\}$, and define $M:=\underset{z\in M_1}\bigcup{z + B_\frac{\varepsilon}{2}}.$

Then $M$ is an $L^0$-convex $L^0$-absorbing neighbourhood of $0\in E$, which is relatively stable (this is because $M_1$ and $B_\frac{\varepsilon}{2}$ are relatively stable and the sum of $E$ preserves the relative stability). Besides, for every $z \in M$ there exists $C_z \in \mathcal{A}^+$ with $\Vert x-z\Vert\geq\frac{\varepsilon}{2}$ on $C_z$. So that $x \notin \overline{M}$.
%\begin{equation}
%\label{ineqII}
%\begin{array}{cc}
%\left\Vert X - Z \right\Vert > \frac{\varepsilon}{2} & \textnormal{ on } C_Z. 
%\end{array}
%\end{equation}

%So that $X \notin \overline{M}$.

Thus, by Proposition \ref{closureGauge} we have that there exists $C\in \mathcal{A}^+$ such that
\begin{equation}
\label{ineqII}
\begin{array}{cc}
p_M (x) > 1 & \textnormal{ on } C,
\end{array}
\end{equation}

where $p_M$ is the gauge function of $M$.

Further, given $\eta,\xi \in L^0$ with $1_C \eta x =1_C \xi x$, it holds that $\eta=\xi$ on $C$. 

Indeed, define $A=(\eta-\xi\geq 0)$
\[
1_C |\eta-\xi| p_M (x)\leq p_M (1_C |\eta-\xi|x)=p_M ((1_{A}-1_{A^c})1_C (\eta-\xi)x)=p_M(0)=0.
\]
In view of (\ref{ineqII}), we conclude that $\eta=\xi$ on $C$.  

Then, we can define the following $L^0$-linear mapping 
\[
\begin{array}{cc}
\mu_0 : & span_{L^0} \{x\}\longrightarrow L^0\\
& \mu_0 (\eta x):=\eta 1_{C} p_M(x).
\end{array}
\]

In addition, we have that
\[
\begin{array}{cc}
\mu_0 (z) \leq p_M (z) & \textnormal{ for all } z\in span_{L^0}\{x\}.
\end{array}
\] 

Thus, by Theorem \ref{HahnBanachI}, $\mu_0$ extends to an $L^0$-mapping application $\mu$ defined on $E$ such that 
\[
\begin{array}{cc}
\mu (z) \leq p_M (z) & \textnormal{ for all } z\in E.
\end{array}
\]
 
Since $M$ is a neighborhood of $0\in E$, by Proposition \ref{closureGauge}, the gauge function $p_M$ is continuous on $E$. Hence $\mu$ is a continuous
$L^0$-linear function defined on $E$.

Furthermore, we have that
\[
\underset{z \in M_1} \esssup \mu (z) \leq \underset{z \in M} \esssup \mu (z) \leq
\]
\[
\begin{array}{cc}
\leq \underset{z \in M} \esssup p_M (z) \leq 1 < p_M (x)=\mu(x)  & \textnormal{ on }C. 
\end{array}
\]
Therefore, $x$ cannot be a weak accumulation point of $M_1$ contrary to the hypothesis of $x_\gamma$ converging weakly to $x$.
\end{proof}

We have the following corollaries:

\begin{cor}
\label{prop: closedw}
Let $(E,\left\Vert \cdot\right\Vert)$ be an $L^0$-normed module whose sum preserves the relative stability, and let $K\subset E$  be $L^0$-convex and relatively stable, we have that the closure in norm, the closure in the weak topology, and the closure in the stable weak topology are the same, i.e. 
\[
\overline{K}^{\Vert \cdot \Vert} = \overline{K}^{\sigma(E,E^*)}=\overline{K}^{\sigma(E,E^*)_{cc}}.
\]
\end{cor}
\begin{proof}

We have that the norm topology is finer than the stable weak topology, and the latter is finer than the weak topology. Thus we obtain
\[
\overline{K}^{\Vert \cdot \Vert}\subset \overline{K}^{\sigma(E,E^*)_{cc}} \subset \overline{K}^{\sigma(E,E^*)}.
\]
On the other hand, for given $x\in \overline{K}^{\sigma(E,E^*)}$, let us choose a net $\{x_\gamma\}\subset K$ which converges to $x$. Due to Theorem \ref{thm: Mazur},  $K$ being $L^0$-convex and relatively stable, for each $\varepsilon\in L^0_{++}$, we can find 
\[
z_\varepsilon\in \overline{co}^{cc}_{L^0}\{x_\gamma \:; \:\gamma\in \Gamma \}\subset K,
\]  
such that $\Vert z_\varepsilon - x\Vert\leq \varepsilon$.

This means that $x\in \overline{K}^{\Vert \cdot \Vert}$.
\end{proof}

Then, from now on, provided that $E$ has sum which preserves the relative stability or, more particularly, if $E$ is stable, for any $L^0$-convex subset $K$ is relatively stable, we will denote the topological closure by $\overline{K}$ without specifying whether the topology is either weak, stably weak or strong.
 
Let us recall some notions introduced in \cite{key-8}:

\begin{defn}
Let $E[\mathscr{T}]$ be a topological $L^0$-module. A function $f:E\rightarrow\bar{L}^{0}$ is called proper if $f(E)\cap L^0 \neq \emptyset$ and $f>-\infty$; it is said to be $L^0$-convex if $f (\eta x_1 + (1 − \eta)x_2) \leq \eta f (x_1) + (1 − \eta)f (x_2)$ for all $x_1,x_2 \in E$ and
$\eta \in L^0$ with $0 \leq \eta \leq 1$; it said to have the local property if $1_A f(x)=1_A f(1_A x)$ for $A\in\mathcal{A}^+$, and $x\in E$; and finally, $f$ is called lower semicontinuous if the level set $V\left(\eta\right)=\left\{ x\in E\:;\: f\left(x\right)\leq \eta\right\}$ is closed for all $\eta\in L^0$.
\end{defn}

\begin{cor}
\label{cor: lowerSemCont}
Let $(E,\left\Vert \cdot\right\Vert)$ be an $L^0$-normed module whose sum preserves the relative stability, and let $f:E\rightarrow \bar{L}^0$  be a proper $L^0$-convex function. If $f$ is continuous, then $f$ is lower semicontinuous with the weak topology. 
\end{cor}

\begin{proof}
It is a known fact that, if $f$ is $L^0$-convex, then it has the local property (see \cite[Theorem 3.2]{key-8}).

Being $f$ $L^0$-convex and with the local property, we have that $V(\eta)$ is $L^0$-convex and has is relatively stable.

Since $f$ is continuous, we have that $V(\eta)$ is norm closed, and due to  Corollary \ref{prop: closedw}, it is weakly closed as well.

\end{proof}

\section{A randomized version of Krein-\v{S}mulian theorem}
\label{sec-3}

In this section we provide a generalization of the classical theorem of Krein-\v{S}mulian in the context of the $L^0$-normed modules. We will use arguments of completeness and the randomized bipolar theorem (see \ref{thm: bipolar}). In the Appendix we recall the notion of polar and results related.

Before proving the main result of this section, let us introduce some necessary notions.
 
Let $(E,\Vert\cdot\Vert)$ be an $L^0$-normed module. Given $\varepsilon \in L^0_{++}$ define
\[
W_\varepsilon := \left\{ (A,B) \:; \: A,B \subset E,\textnormal{ }  \Vert x - y \Vert \leq \varepsilon \textnormal{ for all } x\in A , y\in B \right\}.
\]      

Then $\mathcal{W}:=\{W_\varepsilon \:; \: \varepsilon \in L^0_{++} \}$ is the base of an uniformity on $E$.

Let $\mathcal{G}$ be the set of all nonempty closed subsets of $E$.

A net $\{A_\gamma\}_{\gamma\in\Gamma}$ in $\mathcal{G}$ is said to converge to $A \in \mathcal{G}$ if for each $\varepsilon\in L^0_{++}$ there is some $\gamma_\varepsilon$ with $(A_{\gamma},A)\in W_\varepsilon$ for all $\gamma \geq \gamma_\varepsilon$. Note that $A$ is unique. 

$\{A_\gamma\}_{\gamma\in\Gamma}$ is said to be Cauchy if for each $\varepsilon\in L^0_{++}$ there is some $\gamma_\varepsilon$ with $(A_{\alpha},A_{\beta})\in W_\varepsilon$ for all $\alpha,\beta \geq \gamma_\varepsilon$.

Further, $E$ is said to be complete if every Cauchy net $\{x_\gamma\}_{\gamma\in\Gamma}\subset E$ converges to some $x\in E$. 

\begin{prop}
\label{prop: KS}
If $(E,\Vert\cdot\Vert)$ is complete, then every Cauchy net $\{A_\gamma\}_{\gamma\in\Gamma}$ in $\mathcal{G}$ converges. Moreover, if $\{A_\gamma\}_{\gamma\in\Gamma}$ is decreasing, then it converges to $A:=\bigcap_{\gamma\in\Gamma} A_{\gamma}$.

\end{prop}
\begin{proof}

If $\{A_\gamma\}_{\gamma\in\Gamma}$ is Cauchy we have that for every $\varepsilon\in L^0_{++}$ there exists $\gamma_\varepsilon$ such that
\begin{equation}
\label{CauchyNet}
\Vert x - y \Vert \leq \varepsilon \textnormal{ for all }x\in A_\alpha, y\in A_\beta \textnormal{ with } \alpha,\beta\geq \gamma_\varepsilon 
\end{equation}
In particular, every net $\{x_\gamma\}_{\gamma\in\Gamma}$ with $x_\gamma\in A_\gamma$ is Cauchy and therefore converges as $E$ is complete.

Define $A:=\{ \lim x_\gamma \:; \: x_\gamma\in A_\gamma \textnormal{ for all }\gamma\in\Gamma\}$. Let us show that $\{A_\gamma\}_{\gamma\in\Gamma}$ converges to $\overline{A}$.

Indeed, given $x\in A$ there exists a net $\{x_\gamma\}_{\gamma\in\Gamma}$ with $x_\gamma\in A_\gamma$ which converges to $x$. Then
by (\ref{CauchyNet}) it holds that
\[
\Vert x_\gamma - y \Vert \leq \varepsilon \textnormal{ for all }y\in A_\alpha, \textnormal{ with } \alpha,\gamma\geq \gamma_\varepsilon.
\]
By taking limits in $\gamma$
\[
\Vert x - y \Vert \leq \varepsilon \textnormal{ for all }y\in A_\alpha, \textnormal{ with } \alpha \geq \gamma_\varepsilon.
\]
Since $x\in A$ is arbitrary, it holds that $(\overline{A},A_\alpha) \in W_\varepsilon$ for all $\alpha \geq \gamma_\varepsilon$, and the proof is complete.

\end{proof}

Finally, let us prove the randomized version of the Krein-\v{S}mulian theorem

\begin{thm} 
\label{thm: Krein-Smulian}
[Randomized version of the Krein-\v{S}mulian theorem]
Let $(E,\Vert\cdot\Vert)$ be a complete $L^0$-normed module which is stable with uniqueness and let $K\subset E^*$ be $L^0$-convex and relatively stable. Then the following statements are equivalent:
\begin{enumerate}
	\item $K$ is weak-$*$ closed;  
	\item $K\cap \left\{ z \in E^* \:; \: \Vert z \Vert^* \leq \varepsilon \right\}$ is weak-$*$ closed for each $\varepsilon \in L^0_{++}$.
\end{enumerate}
\end{thm}
\begin{proof}

$1 \Rightarrow 2$: It is clear because $\left\{ z \in E^* \:; \: \Vert z \Vert^* \leq \varepsilon \right\}=B_{\varepsilon}^o$ is weak-$*$ closed.

$2 \Rightarrow 1$: For each $\varepsilon\in L^0_{++}$ we have that ${B_\varepsilon}^o \cap K$ is weak-$*$ closed. Then the net $\{ ({B_\varepsilon}^o \cap K)^o \}_{\varepsilon\in L^0_{++}}$ is Cauchy. Indeed, for $\delta , \delta' \leq \varepsilon/2$, by using the properties of the polar (see \ref{prop: polar}) and the bipolar theorem  (see \ref{thm: bipolar}):  
\[ 
({B_\delta}^o \cap K)^o+ B_{\varepsilon} \supset \overline{({B_\delta}^o \cap K)^o+B_{\varepsilon/2}}=(({B_\delta}^o \cap K)^o+B_{\varepsilon/2})^{oo}\supset
\]     
\[
\supset (({B_\delta}^o \cap K)^{oo}\cap B_{\varepsilon/2}^o)^o=({B_\delta}^o \cap K\cap B_{\varepsilon/2}^o)^o\supset ({B_{\delta'}}^o \cap K)^o.
\]
Since the net is decreasing, due to Proposition \ref{prop: KS}, it converges to $C:=\bigcap_{\varepsilon\in L^0_{++}} ({B_{\varepsilon}}^o \cap K)^o$. 

Let us see that $K=C^o$. Indeed, $C\subset({B_{\varepsilon}}^o \cap K)^o$ so ${B_{\varepsilon}}^o \cap K\subset C^o$ for all $\varepsilon\in L^0_{++}$ and therefore $K\subset C^o$.

Let $\varepsilon, r\in L^0_{++}$ with $r>1$. Since the net converges to $C$, there is $\delta\in L^0_{++}$ with 
\[
({B_{\delta}^o \cap K)^o \subset C + (r-1)B_{\varepsilon}} \subset (C\cup B_{\varepsilon})^{oo} + (r-1)(C\cup B_{\varepsilon})^{oo} =r(C\cup B_{\varepsilon})^{oo}
\]  
and by taking polars, it follows that 
\[
[r(C\cup B_{\varepsilon})^{oo}]^o=\frac{1}{r}(C\cup B_{\varepsilon})^o \subset (B_{\delta}^o \cap K)^{oo}=B_{\delta}^o \cap K
\]
and thus $C^o \cap B_{\varepsilon}^o \subset r(B_{\delta}^o \cap K)$ for all $r\in L^0_{++}$, $r>1$.
Therefore
\[
B_{\varepsilon}^o \cap C^o \subset \underset{r\in L^0_{++},r<1} \bigcap r(B_{\delta}^o \cap K)\subset \overline{B_{\delta}^o \cap K}=B_{\delta}^o \cap K\subset K,
\]
 and the proof is complete.

\end{proof}

%%%%%%%%%%%%%%%%%%%%%%
\begin{rem}
After writing the first version of this paper, the author knew from Asgar Jamneshan of the work \cite{key-21}. So, the author is grateful to Asgar Jamneshan for pointing out this reference.
  %
%Namely, K.-T. Eisele and S. Taieb \cite{key-18} proved a version of Krein-\v{S}mulian theorem under the less general structure of $L^\infty$-modules by using a different strategy-proof. 
%
 S. Drapeau et al. \cite{key-21}, in an abstract level, introduced the algebra of conditional sets, in which strong stability properties are assumed on all structures. This framework is related to the $L^0$-theory, in the sense that if we assume stability properties with uniqueness on an $L^0$-module, we obtain a conditional set. In this context, it was proved a version of Krein-\v{S}mulian theorem with a different strategy-proof. However we would like to emphasize that we work under weaker stability properties. For instance, the relative stability assumes neither the existence nor the uniqueness of a concatenation. Besides, briefly speaking, topologies used here does not need to be a 'stable family' in the sense proposed by \cite{key-21}. For instance, the weak topologies employed for the randomized Mazur's lemma and Krein-\v{S}mulian theorem proved here are not necessarily stable families. In fact, Example \ref{ex: weakTop} shows that the weak topology employed is not stable, and is strictly weaker than its stable version.

Also, the randomized Mazur's lemma proved here applies to  modules of the type provided in Proposition \ref{prop: sickModules}, which are not stable but their sum preserves the relative stability. Such modules fall out of the scope of \cite{key-21}; however, model spaces  with some lack of stability could be eventually subject of financial applications.          
\end{rem}
%%%%%%%%%%%%%%%%%%%%%%

\begin{appendix}

\section{Appendix}

Let us collect some important results from the theory of locally $L^0$-convex modules:
		
\begin{thm}
\cite[Hahn–Banach extension theorem]{key-8}
\label{HahnBanachI} Let $E[\mathscr{T}]$ be a locally $L^0$-convex module. 
Consider an $L^0$-sublinear function $p : E → L^0$, an $L^0$-submodule $C$ of $E$ and an $L^0$-linear function $\mu : C → L^0$ such that
\[
\begin{array}{cc}
\mu(x)\leq  p(x) & \textnormal{ for all }x\in C
\end{array}
\]

Then $\mu$ extends to an $L^0$-linear function $\tilde{\mu} : E → L^0$ such that $\tilde{\mu}(x)\leq  p(x)$ for all $x \in E$.

\end{thm}

%%%%%%%%%%%%%%%%%%%%%
%%%%%%%%%%%%%%%%%%%%

\begin{defn}
Given $A\subset E$, we define the polar of $A$ the subset $A^o$ of $E^*$ given by 
\[
A^o:=\{y\in E^* \:; \: |\langle x,y\rangle|\leq 1, \textnormal{ for all }x\in A\}.
\] 
\end{defn}

\begin{prop}
\label{prop: polar}
Let $E[\mathscr{T}]$  be a locally $L^0$-convex module, $D,D_i \subset E$ for $i\in I$, then:
\begin{enumerate}
\item $D^o$ is $L^0$-convex, weak-$*$ closed and relatively stable. 
%\item $D^{o}=\overline{\overline{co_{L^{0}}}^{cc}}\left\{ D\cup\left\{ 0\right\} \right\} $, the topological closure of the countable concatenation closure of the $L^0$-convex hull.
\item $0\in D^o$, $D\subset D^{oo}$. If $D_1\subset D_2$, then $D_2^o\subset D_1^o$.
\item For $\varepsilon \in L^0_{++}$, we have that $\left({\varepsilon} D \right)^o=\frac{1}{\varepsilon} D ^o$.
\item $\left(\bigcup_{i\in I}D_{i}\right)^{o}=\bigcap_{i\in I}D_{i}^{o}$
%\item $A^o \cap B^o \subset (A+B)^o$
\end{enumerate}
\end{prop}
\begin{proof}
Let us see that $D^o$ is relatively stable.

Given $z$ a concatenation of $\{z_n\}\subset D^o$ along $\{A_k\}\in\Pi(\Omega)$, we have that 
\[
|\langle x,y\rangle|=\left(\underset{n\in\N}\sum 1_{A_n}\right) |\langle x,y\rangle| = \underset{n\in\N}\sum 1_{A_n} |\langle x,\sum 1_{A_n} x\rangle| =
\]
\[
\underset{n\in\N}\sum 1_{A_n} |\langle x,\sum 1_{A_n} z_n\rangle| =\underset{n\in\N}\sum 1_{A_n} |\langle x, z_n\rangle| \leq 1.
\]
The rest is analogue to the known proofs for locally convex spaces.
\end{proof}

\begin{thm}
\cite[Theorem 3.26]{key-12}
\label{thm: bipolar}
Let $E[\mathscr{T}]$  be a locally $L^0$-convex module which is stable with uniqueness, and suppose $D \subset E$, then:
\[
D^{oo}=\overline{\overline{co_{L^0}}^{cc}}(D\cup\{0\})
\]
\end{thm}

\end{appendix}


\begin{thebibliography}{1} % Beamer does not support BibTeX so references must be inserted manually as below
\bibitem{key-2} P. Artzner, F. Delbaen, J.-M. Eber, D. Heath, Coherent risk measures, Math. Finance 9 (1999) 203–228.

\bibitem{key-7} J. Bion-Nadal, Conditional risk measure and robust representation of convex conditional risk measures, CMAP Preprint 557 (2004).

\bibitem{key-16} D. L. Cohn, Measure theory, Vol. 1993. Boston: Birkhäuser, 1980.

\bibitem{key-5} F. Delbaen, Monetary utility functions, University of Osaka, Osaka, 2011.

\bibitem{key-20} F. Delbaen, W. Schachermayer, The mathematics of arbitrage, Springer Science \& Business Media, 2006.

\bibitem{key-6} K. Detlefsen, G. Scandolo, Conditional and dynamic convex risk measures, Finance Stoch. 9 (2005) 539–561.

\bibitem{key-21} S. Drapeau, A. Jamneshan, M. Karliczek, \& M. Kupper. The algebra of conditional sets and the concepts of conditional topology and compactness. Journal of Mathematical Analysis and Applications, 437(1) (2016) 561-589.

\bibitem{key-18} K.-T. Eisele, S. Taieb, Weak topologies for modules over rings of bounded random variables, J. Math. Anal. Appl. 421.2 (2015) 1334-1357.


\bibitem{key-8} D. Filipovic, M. Kupper, N. Vogelpoth, Separation and duality in locally $L^0$-convex module, J. Funct. Anal. 256 (2009) 3996–4029.

\bibitem{key-9} D. Filipovic, M. Kupper, N. Vogelpoth, Approches to conditional risk, SIAM J. Financial Math. (2012) 402-432.


\bibitem{key-3} H. Föllmer, A. Schied, Convex measures of risk and trading constraints, Finance Stoch. 6.4 (2002) 429-447.

\bibitem{key-1} H. Föllmer, A. Schied, Stochastic finance: an introduction in discrete time, Walter de Gruyter, 2011.

\bibitem{key-4} M. Fritelli, E. Gianin Putting order in risk measures, J. Banking Finance, 26 (2002) 1473–1486


\bibitem{key-10} T. Guo, The relation of Banach–Alaoglu theorem and Banach–Bourbaki–Kakutani–\v{S}mulian Theorem in complete random normed modules to stratification structure, Sci. China Ser. A 51 (2008) 1651–1663.

\bibitem{key-32} T. Guo, X. Chen. Random duality. Science in China Series A: Mathematics, 52(10), 2084-2098 (2009).

\bibitem{key-22} T. Guo, T., G. Shi. The algebraic structure of finitely generated $L^0(F, K)$-modules and the Helly theorem in random normed modules. Journal of Mathematical Analysis and Applications, 381(2), 833-842 (2011).


\bibitem{key-15} T. Guo, Relations between some basic results derived from two kids of topologies for a random locally convex module, J. Funct. Anal. 258 (2010) 3024–3047

%\bibitem{key-40}  T. Guo, S. Zhao, X. Zeng. Random convex analysis (I): separation and Fenchel-Moreau duality in random locally convex modules. arXiv preprint:1503.08695 (2015).
%
%\bibitem{key-13} T. Guo, S. Zhao, X. Zeng.  Random convex analysis (II): continuity and subdifferentiability theorems in $ L^{0} $-pre-barreled random locally convex modules. arXiv preprint:1503.08637 (2015).


\bibitem{key-35} T. Guo, H. Xiao, X. Chen. A basic strict separation theorem in random locally convex modules. Nonlinear Analysis: Theory, Methods \& Applications 71.9 (2009): 3794-3804.

\bibitem{key-12} T. Guo, S. Zhao, X. Zeng, On random convex analysis-the analytic foundation of the module approach to conditional risk measures, arXiv preprint arXiv:1210.1848 (2012).

\bibitem{key-17} T. Guo, S. Zhao, X. Zeng, The relations among the three kinds of conditional risk measures, Sci. China Math. 57.8 (2014): 1753-1764.

\bibitem{key-34}  T. Guo, S. Zhao, X. Zeng. Random convex analysis (I): separation and Fenchel-Moreau duality in random locally convex modules. arXiv preprint arXiv:1503.08695 (2015).

\bibitem{key-36} T. Guo, S. Zhao, X. Zeng.  Random convex analysis (II): continuity and subdifferentiability theorems in $ L^{0} $-pre-barreled random locally convex modules. arXiv preprint arXiv:1503.08637 (2015).

\bibitem{key-31} R. Haydon, M. Levy, Y. Raynaud, Randomly Normed Spaces, Hermann,
Paris, 1991.

\bibitem{key-14} J. L. Kelley, I. Namioka, et al., Linear topological spaces, Springer Berlin Heidelberg, 1963.

\bibitem{key-50} J. Orihuela, and J. M. Zapata. Stability in locally $L^0$-convex modules and a conditional version of James’ compactness theorem. Journal of Mathematical Analysis and Applications, 452(2) (2017) 1101–1127.


\bibitem{key-33} Rüschendorf, Ludger. Mathematical risk analysis. Springer Ser. Oper. Res. Financ. Eng. Springer, Heidelberg (2013).

%\bibitem{key-22} D. G. Wright, Tychonoff's theorem, Proc. Am. Math. Soc. (1994) 985-987.

\bibitem{key-19} M. Wu, T. Guo, A counterexample shows that not every locally $L^0$-convex topology is necessarily induced by a family of $L^0$-seminorms, arXiv preprint arXiv:1501.04400 (2015).

\bibitem{key-23}  K. Yosida, Functional Analysis, Springer-Verlag, Berlin 1980.

\bibitem{key-11} J. M. Zapata. On the Characterization of Locally $L^0$-Convex Topologies Induced by a Family of $L^0$-Seminorms, J. Convex Anal. 24(2) (2017) 383–391.   




\end{thebibliography}
\end{document}